\documentclass{amsart}

\usepackage{amsmath}
\usepackage{amssymb}
\usepackage{amscd}
\usepackage[T1]{fontenc}

\title[Equivariant Euler-Poincar\'e Characteristic]
 {On equivariant Euler-Poincar\'e characteristic in sheaf cohomology}

\author[S. Kionke]{Steffen Kionke}
\address{Steffen Kionke\\
         Max-Planck-Institut f\"ur Mathematik\\
         Vivatsgasse 7\\
         53111 Bonn\\ Germany}
\email{skionke@mpim-bonn.mpg.de}
\author[J. Rohlfs]{J\"urgen Rohlfs}
\address{J\"urgen Rohlfs\\
Katholische Universit\"at Eichst\"att-Ingolstadt\\
Ostenstrasse 26\\
Kollegiengeb\"aude I/Bau B\\
85072 Eichst\"att\\ Germany}
\email{juergen.rohlfs@ku-eichstaett.de}
\thanks{The first author wants to thank the MPIM in Bonn for their hospitality and support.}

\theoremstyle{plain}
\newtheorem{lemma}{Lemma}
\newtheorem*{theorem}{Theorem}
\newtheorem{corollary}{Corollary}

\theoremstyle{definition}

\newtheorem*{remark}{Remark}

\newtheorem*{definition*}{Definition}

\DeclareMathOperator{\ind}{ind}

\DeclareMathOperator{\Hom}{Hom}

\DeclareMathOperator{\Aut}{Aut}
\DeclareMathOperator{\derivedR}{R}
\DeclareMathOperator{\Der}{D}

\DeclareMathOperator{\cla}{Cl}
\DeclareMathOperator{\tr}{tr}

\providecommand{\RR}{\underline{\derivedR}}
\providecommand{\calF}{\mathcal{F}}

\providecommand{\calE}{\mathcal{E}}
\providecommand{\calF}{\mathcal{F}}
\providecommand{\calL}{\mathcal{L}}
\providecommand{\com}[1]{#1^{\bullet}}
\providecommand{\Gzero}{\mathsf{G_0}}
\providecommand{\Sh}{\mathsf{Sh}}

\providecommand{\isomorph}{\stackrel{\simeq}{\longrightarrow}}

\providecommand{\bbR}{\mathbb{R}}
\providecommand{\bbQ}{\mathbb{Q}}
\providecommand{\bbZ}{\mathbb{Z}}
\providecommand{\bbF}{\mathbb{F}}
\providecommand{\bbC}{\mathbb{C}}

\providecommand{\cF}{\textup{(\textbf{F})} }
\providecommand{\cFone}{\textup{(\textbf{F1})} }
\providecommand{\cFtwo}{\textup{(\textbf{F2})} }

\begin{document}

\begin{abstract}
Let $X$ be a topological Hausdorff space together with a continuous action of a finite group $G$. Let $R$
be the ring of integers of a number field~$F$. Let $\calE$ be a $G$-sheaf of flat $R$-modules over $X$ and let
$\Phi$ be a $G$-stable paracompactifying family of supports on $X$. We show that under some natural cohomological
finiteness conditions the Lefschetz number of the action of $g \in G$ on the cohomology
$ \com{H}_\Phi(X,\calE) \otimes_{R} F $ equals the Lefschetz number of the $g$-action on
$ \com{H}_{\Phi|X^G}(X^g, \calE_{|X^g}) \otimes_{R} F $, where $X^g$ is the set of fixed points of $g$
in $X$. More generally, the class $\sum_j (-1)^j [H^j_\Phi (X,\calE) \otimes_R F]$ in the character group
equals a sum $\sum_{[H]} \sum_{\lambda \in \widehat{H}_F} m_\lambda [\ind^G_H (V_\lambda)] $ of representations induced from
irreducible $F$-rational representations $\:   V_\lambda  \:$ of $\:  H \:,$ where $[H]$ runs in the set of
$G$-conjugacy classes of subgroups of $G$. The integral coefficients $m_\lambda$ are explicitly
determined.
\end{abstract}

\maketitle

\section{Introduction and main results}

The most elementary classical version of the Lefschetz fixed point formula says that the Lefschetz  number
$\calL(g)$ of a simplicial automorphism $g$ of finite order on a finite simplicial complex $X$ equals the Euler-Poincar\'e
characteristic of the fixed point set $X^g \subset X $ of the $g$-action. Here $\calL(g)$
is computed  on $ H^\ast (X,\bbQ)$. Brown \cite{Brown1982} (based on Zarelua \cite{Zarelua1969}) and independently Verdier \cite{Verdier1973}
have extended this formula to more general spaces under the assumption of cohomological finiteness conditions.
Verdier uses cohomology with compact supports.

The objective of this paper is to generalize this Lefschetz fixed point formula to Hausdorff spaces with a
continuous action of a finite group $G$ and to cohomology of $G$-sheaves with a paracompactifying family of supports.
For applications of the Lefschetz fixed point formula
to cohomology of arithmetic groups see e.g.~\cite{Rohlfs1990}.

\subsection{Notation}
 Throughout $F$ denotes an algebraic number field and $R$ denotes its ring of integers.
Let $G$ be a finite group, then $\Gzero(F[G])$ denotes the Grothendieck group of finitely generated $F[G]$-modules.
For every subgroup $H$ of $G$ there is the induction homomorphism $\ind_H^G: \Gzero(F[H]) \to \Gzero(F[G])$
 which maps $[M]$ to $[F[G]\otimes_{F[H]}M]$.
 For $ g \in G $ the trace of the $g$-action induces a morphism
$\tr(g): \Gzero(F[G]) \longrightarrow R$.

Let $Y$ be a Hausdorff space and let $\Phi$ be a family of supports on $Y$.
Let $k$ be a ring and $\calE$ be a sheaf of left $k$-modules on $Y$. If the cohomology $\com{H}_\Phi(Y,\calE)$
is finitely generated as $k$-module, then we say that the triple $(Y,\Phi,\calE)$ is of \emph{finite type} with respect to $k$.

Given a sheaf $\calE$ of $R$-modules on $Y$, we write
$\chi_\Phi(Y,\calE; F)$ for the Euler-Poincar\'e characteristic of the graded $F$-vectorspace
$\com{H}_\Phi(Y,\calE)\otimes_R F$ whenever it is \mbox{finite} dimensional.
Similarly, if a finite group $G$ acts continuously on $Y$ and $\calE$ is $G$-equivariant, then we denote the Euler-Poincar\'e characteristic
of the graded $F[G]$-module $\com{H}_\Phi(Y,\calE)\otimes_R F$ in the Grothendieck
group $\Gzero(F[G])$ by $\chi_\Phi(Y,\calE; F[G])$.
 The image of $\chi_\Phi (Y, \calE; F[G]) $ under the morphism $\tr(g)$ is the
\emph{Lefschetz number} $ \calL_\Phi (g, \calE; F) = \sum^\infty_{j = 0} (-1)^j \tr (g| H^j _\Phi (Y ,\calE) \otimes_R F) \:.$

\subsection{Statement of results} 
Denote by $X$ a topological Hausdorff space together with a continuous action of a finite group $G$.
We fix a paracompactifying $G$-stable family of supports $\Phi$ on $X$.
For a subgroup $H \leq G$ we denote the normalizer of $H$ in $G$ by $N_G(H)$ or $N(H)$.
We write $X^H$ for the set of points in $X$ which are fixed by $H$ and we write $X_H$ for
the set of points in $X$ whose stabilizer is exactly the group $H$. Note that $X^H$ is closed in $X$ and $X_H$ is open in $X^H$.
Let $\cla(G)$ be the set of conjugacy classes of subgroups of $G$. The paracompactifying
family $\Phi$ induces paracompactifying families on the locally closed subspaces
$ X_H, X^H, X_C := \bigcup_{H \in C} X_H $ for $C \in \cla(G)$  and on the
quotient space $ X_H/ N(H) $. For simplicity these families will be
denoted also by $\Phi$.

By $ \widehat{H}_F$ we denote the set of equivalence classes of
irreducible representations of $H$ on finite dimensional $F$-vectorspaces. If $\lambda \in \widehat{H}_F$,
we write $V_\lambda$ for a representative of $\lambda$. We define 
$\deg V_\lambda := \dim_F \bigl(\Hom_{F[H]} (V_\lambda, V_\lambda )\bigr)$. 
\vspace{0.8em}

  Let $\calE$ be a $G$-sheaf of $R$-modules on $X$.
  We say that $\calE$ satisfies the finiteness condition \cF  if
  for every subgroup $H$ of $G$ the following hold:
  \begin{enumerate}
   \item  The triple $(X_H,\Phi, \calE_{|X_H}) $ is of finite type w.r.t.~$R$, and
   \item  for any $\lambda \in \widehat{N(H)}_F$
    there is an $N(H)$-invariant lattice $L_\lambda$ in $V_\lambda$ such that the triple
    $(X_H,\Phi,\Hom_{R[H]}\bigl(L_\lambda, \calE_{|X_H} \bigr))$ is of finite type w.r.t.~$R$.

  \end{enumerate}
We comment on this condition in section \ref{sec:Comments}.

\begin{theorem} Let $X$, $G$, $\Phi$ be as above and assume that the cohomological
$\Phi$-dimension of $X$ is finite. Let $ \calE$ be a $G$-sheaf of flat $R$-modules such that
condition \cF holds. Then
\begin{align*}
         \chi_\Phi\bigl(X,\calE; F[G]\bigr) &=
        \sum_{[H] \in \cla(G)} \frac{|H|}{|N(H)|} \ind_H^G\Bigl(\chi_{\Phi}\bigl(X_H, \calE_{|X_H}; F[H]\bigr)\Bigr)\\
         & = 
        \sum_{[H] \in \cla(G)} \sum_{\lambda \in \widehat{H}_F}\frac{|H|\cdot e(\lambda)}{|N(H)| \cdot \deg(V_\lambda)} \ind_H^G\bigl([V_\lambda\bigr]).
\end{align*}
   where $e(\lambda)$ denotes the Euler characteristic
   $\chi_{\Phi}\bigl(X_H, \Hom_{R[H]}(M_\lambda, \calE_{|X_H}); F\bigr)$ for any
    $H$-stable $R$-lattice $M_\lambda \subset V_\lambda$. 
\end{theorem}
A proof of the theorem will be given in the next section.

\begin{corollary}\label{cor:letg}
       Let $G$ be the finite cyclic group generated by an element $g$. Under the assumptions of the theorem we obtain
       an equality of Lefschetz numbers
       \begin{equation*}
          \calL_\Phi\bigl(g,\calE; F\bigr) = \calL_{\Phi}\bigl(g, \calE_{|X^G}; F \bigr).
       \end{equation*}
   \end{corollary}
\begin{proof}
 Use  that for
 $ G = \langle g \rangle $ we have $ \tr (g | \ind^G_H V) = 0 \:$ for all finite dimensional $ F[H]$-modules $V$ if
 $ H \neq G $ and that $ X_G = X^G $.
\end{proof}

\section{Proof of the  Theorem}

This section is devoted to the proof of the  theorem.
We begin with the following general Lemma.
\begin{lemma}\label{lem:EulerCharModp}
 Let $p$ be a prime number.
 Let $\calE$ be a sheaf of abelian groups on $X$ and assume that the stalks of $\calE$ have no $p$-torsion.
 If the triple $(X,\Phi,\calE)$ is of finite type (w.r.t.~$\bbZ$) then the triple
 $(X,\Phi,\calE\otimes_\bbZ \bbF_p)$ is of finite type w.r.t.~$\bbF_p$.
 Moreover we have
 \begin{equation*}
      \chi_\Phi(X,\calE; \bbQ) =  \chi_\Phi(X,\calE\otimes_\bbZ \bbF_p; \bbF_p).
 \end{equation*}
Here $\bbF_p$ denotes the finite field with $p$ elements.
\end{lemma}
\begin{proof}
 Since $\calE$ is torsion-free, there is a short exact sequence of sheaves on $X$:
 \begin{equation*}
    0 \longrightarrow \calE \stackrel{p}{\longrightarrow} \calE \longrightarrow \calE\otimes_\bbZ \bbF_p \longrightarrow 0.
 \end{equation*}
  Consider the associated long exact sequence. We deduce that $(X,\Phi,\calE\otimes_\bbZ \bbF_p)$ is of finite type.
  Further, write $H^i_\Phi(X,\calE) \cong \bbZ^{b_i} \oplus P^i \oplus T^i$, where $P^i$ is the subgroup of elements whose order is a power of $p$
  and $T^i$ is the subgroup of elements of finite order prime to $p$.
  Let $P^i_p$ denote the elements of order exactly $p$ and let $r_i = \dim_{\bbF_p} P^i_p$.
  From the long exact sequence we obtain short exact sequences
  \begin{equation*}
     0 \longrightarrow H^i_\Phi(X,\calE)\otimes_\bbZ\bbF_p \longrightarrow H^i_\Phi(X,\calE\otimes\bbF_p)\longrightarrow P^{i+1}_p \longrightarrow 0
  \end{equation*}
  for every degree $i$. Thus $\dim_{\bbF_p} H^i_\Phi(X,\calE\otimes\bbF_p) = b_i + r_i + r_{i+1}$ and the second assertion follows
  via alternating summation.
\end{proof}

 Let $\pi: X \to X/G$ be the canonical projection.
 Note that for a sheaf of abelian groups $\calE$ on $X$ there is a canonical isomorphism
\begin{equation*}
  \com{H}_\Phi(X,\calE) \isomorph \com{H}_\Phi(X/G,\pi_*(\calE)).
\end{equation*}

  In general, if $\calF$ is a sheaf of $R[G]$-modules on a space $Y$, then we write $\calF^G$ for the subsheaf of $G$-stable sections,
  i.e.~$\calF^G(U) = \calF(U)^G$.
  Let $\calE$ be a $G$-sheaf of $R$-modules on $X$. We write $\pi_*^G(\calE)$ for $\pi_*(\calE)^G$.
  Note that the
  triple $(X,\Phi, \calE)$ is of finite type w.r.t.~$R$ if and only if it is of finite type w.r.t.~$\bbZ$.
  In this case we simply say that $(X,\Phi, \calE)$ is of finite type.

\begin{lemma}\label{lem:EulerCharCovering}
   Suppose that $G$ is abelian and acts freely on $X$. Let $\calE$ be a flat $G$-sheaf of $R$-modules on $X$ such that
   $(X,\Phi, \calE)$ is of finite type. In this case $(X/G, \Phi, \pi_*^G(\calE))$ is of finite type and
   \begin{equation*}
       \chi_\Phi(X,\calE; F) = |G| \chi_\Phi(X/G, \pi_*^G(\calE); F).
   \end{equation*}
\end{lemma}
\begin{proof}
   First note that $X/G$ has finite $\Phi$-dimension, since cohomological dimension
   is a local property (cf.~ II.~4.14.1 in \cite{Godement1958}) and $\pi$ is a covering map.
   Further, the triple $(X/G, \Phi , \pi_*^G(\calE))$ is of finite type due to the Grothendieck spectral sequence
    \begin{equation*}
       H^p(G, H^q_\Phi(X,\calE)) \implies H^{p+q}_\Phi(X/G,\pi_*^G(\calE))
    \end{equation*}
   which can be obtained for paracompactifying supports just as in \cite[Thm.~5.2.1]{Grothendieck1957}.

   Now we prove the assertion about the Euler characteristic.
   It is easy to check that $[F:\bbQ] \chi_\Phi(X, \calE; F) = \chi_\Phi(X, \calE; \bbQ)$, hence we can assume $R = \bbZ$.
   By induction on the group structure we can assume that $G$ is finite cyclic of prime order $p$.
   The assertion follows from Lemma \ref{lem:EulerCharModp} and a Theorem of E.~E.~Floyd (based on a result of P. A. Smith),
   see \cite[Thm.~19.7]{Bredon1997} or \cite[Thm.~4.2]{Floyd1952}. Here we use that $\pi_*^G(\calE)\otimes_\bbZ \bbF_p = \pi_*^G(\calE\otimes_\bbZ \bbF_p)$.
  Note that we assumed the $\Phi$-dimension of $X$ to be finite, which
   implies in particular that $\dim_{\Phi,\bbF_p} X$ is finite in the notation of \cite{Bredon1997}.
   Further the pull-back sheaf $\pi^* (\pi_*^G(\calE))$ is isomorphic to $\calE$ (see~\cite[p.~199]{Grothendieck1957}).
\end{proof}

  \begin{lemma}
     Let $G$ be a finite group which acts freely on $X$ and let $\calE$ be flat a $G$-sheaf of $R$-modules on $X$.
     We assume that $(X,\Phi, \calE)$ is of finite type. For any $g \in G$ with $g \neq 1$ the Lefschetz number vanishes.
  \end{lemma}
  \begin{proof}
     By taking a finite extension we can assume without loss of generality that $R$ contains all $|G|$-th roots of unity.
     Further we can assume that $G$ is a finite cyclic group. Let $\psi: G \to R^\times$ be a character of $G$.
     We can twist the $G$-sheaf $\calE$ with the character $\psi^{-1}$ to obtain a new $G$-sheaf $\calE\otimes \psi^{-1}$.
     This sheaf is isomorphic to $\calE$ as a sheaf of $R$-modules, but not as $G$-sheaf.
     Further we find that $\pi_*^G(\calE\otimes\psi^{-1})$ is the $\psi$-eigensheaf $\pi_*(\calE)_\psi$ in the
     sheaf $\pi_*(\calE)$ of $R[G]$-modules, this means $\pi_*(\calE)_\psi$ is the subsheaf of
     sections of $\pi_*(\calE)$ which transform with $\psi$ under the action of~$G$. From Lemma \ref{lem:EulerCharCovering}
     we deduce that all the eigensheaves $\pi_*(\calE)_\psi$ have equal Euler characteristic.
     However, since
     \begin{equation*}
        \calL_\Phi(g,\calE; F) = \sum_{\psi \in \widehat{G}} \psi(g) \chi_\Phi(X/G, \pi_*(\calE)_\psi; F)
     \end{equation*}
     the claim follows.
  \end{proof}

We shall frequently use the following Lemma.

\begin{lemma}\label{lem:FixInside}
  Let $\calE$ be a sheaf of $R[G]$-modules on a space $X$ and let $\Phi$ be a system of supports on $X$.
  The inclusion $\calE^G \to \calE$ induces an isomorphism of vectorspaces
  \begin{equation*}
       \com{H}_\Phi(X,\calE^G) \otimes_R F \isomorph \com{H}_\Phi(X,\calE)^G \otimes_R F.
  \end{equation*}
\end{lemma}
\begin{proof}
   Consider the functor $B: \calE \mapsto \calE^G$ from the category $\Sh_X(R[G])$ of sheaves of $R[G]$-modules
   to the category $\Sh_X(R)$ of sheaves of $R$-modules. This functor is left exact and we consider its right derived functor
   \begin{equation*}
      \RR{B}: \Der^+(\Sh_X(R[G])) \to \Der^+(\Sh_X(R)).
   \end{equation*}
 Note that $B$ takes injective sheaves of $R[G]$-modules to
   flabby sheaves (see Corollary to Prop.~5.1.3 in \cite{Grothendieck1957}).
   As in Thm.~5.2.1 in \cite{Grothendieck1957} there is a convergent spectral sequence
   \begin{equation*}
       H^p_\Phi(X,\RR{B}^q(\calE))\otimes_R F \implies H^{p+q}_\Phi(X,\calE)^G\otimes_R F,
   \end{equation*}
    where we use that $F$ is a flat $R$-module.
    In fact, the stalk at $x \in X$ of $\RR{B}^q(\calE)$ is the group cohomology $H^q(G,\calE_x)$ which is purely $|G|$-torsion
    for all $q \geq 1$. Hence the spectral sequence collapses and the claim follows.
\end{proof}

  We obtain a refined version of Verdier's Lemma (cf.~\cite{Verdier1973}).
\begin{lemma}\label{lem:VerdierLemma}
   Let $G$ be a finite group which acts freely on $X$ and let $\calE$ be a flat $G$-sheaf of $R$-modules on $X$.
   We assume that $(X,\Phi, \calE)$ is of finite type.
   In this case we have
    \begin{equation*}
       \chi_\Phi(X, \calE; F[G]) = \chi_\Phi(X/G,\pi^G_*(\calE); F) \cdot F[G].
    \end{equation*}
 In particular, Lemma \ref{lem:EulerCharCovering} holds without the assumption that $G$ is abelian.
\end{lemma}
\begin{proof}
   With the same argument as in Lemma \ref{lem:EulerCharCovering} we see that $(X/G, \Phi, \pi^G_*(\calE))$ is of finite type.
   It suffices to compute the Lefschetz numbers of all elements of $G$ and compare them with the right hand side.
   The vanishing of all Lefschetz numbers for $g \neq 1$ shows that $ \chi_\Phi(X, \calE; F[G])$ is a multiple of the regular representation.
   The coefficient is the Euler characteristic of the graded $F$-vectorspace $\com{H}_\Phi(X,\calE)^G\otimes_{R} F$.
   Since $\com{H}_\Phi(X,\calE) \cong \com{H}_\Phi(X/G,\pi_*(\calE))$, we can use Lemma \ref{lem:FixInside} to deduce the claim.
\end{proof}
\begin{remark}
  Verdier uses the projection formula, the finite tor-amplitude criterion and a famous theorem of Swan
  to obtain this lemma for cohomology with compact supports and constant coefficients. It is possible to extend his approach to the case 
  of families of supports
  using a suitable replacement for the projection formula, see~\cite{Kionke2012}.
\end{remark}

Finally we prove the main theorem. Recall that $X$ is a Hausdorff space with an action of a finite group $G$ and $\Phi$ is a $G$-invariant
paracompactifying system of supports.
  \begin{proof}[Proof of the Theorem]
   Note that for every subgroup $H$ of $G$ the space $Y = X^H$ (resp.~$Y =X_H$) has
   finite $\Phi$-dimension since $Y$ is (locally) closed and $\Phi$ is paracompactifying (cf.~II.~Rem.~4.14.1 in \cite{Godement1958}).
   By condition \cFone the triple
   $(X_H,\Phi, \calE_{|X_H})$ is of finite type for every $H \leq G$.
   For $i=1, \dots, |G|$ we define the closed set
   \begin{equation*}
      X^i = \bigcup_{\substack{H \leq G \\ |H|\geq i}} X^H.
   \end{equation*}
    Then $X^1 = X$ and $X^{|G|} = X^G$ is the set of fixed points.
    An element $x\in X$ is in $X^i\setminus X^{i+1}$ exactly if it has an isotropy group with
    $i$ elements, hence
    \begin{equation*}
       X^i\setminus X^{i+1} = \bigcup_{\substack{H \leq G \\ |H|=i}} X_H =: \bigcup_{\substack{C \in \cla(G)_i}} X_C
    \end{equation*}
     where $\cla(G)_i$ is the set of conjugacy classes of subgroups with $i$ elements.
    Note that these unions are topologically disjoint. Using the long exact sequences of the pairs $(X^i,X^{i+1})$
    with supports in $\Phi$ (cf.~II.~4.10.1 in \cite{Godement1958}) we obtain
    \begin{equation*}
        \chi_\Phi\bigl(X,\calE; F[G]\bigr) = \sum_{C \in \cla(G)} \chi_{\Phi}\bigl(X_C, \calE_{|X_C}; F[G]\bigr).
    \end{equation*}
    Since $X_C$ is the disjoint union $\bigcup_{H \in C} X_H$ and $G$ acts transitively on the components
    we obtain
    \begin{equation*}
        \chi_{\Phi}\bigl(X_C, \calE_{|X_C}; F[G]\bigr) = \ind^G_{N_G(H)}\chi_{\Phi}\bigl(X_H, \calE_{|X_H}; F[N_G(H)]\bigr)
    \end{equation*}
    for any representative $H \in C$.
    We are now in the specific situation where $H$ acts trivially on $X_H$ and $N_G(H)/H$ acts freely on $X_H$.
     For simplicity we write $N$ for the normalizer $N_G(H)$.
  We prove the following identity
  \begin{equation*}
      \chi_{\Phi}\bigl(X_H, \calE_{|X_H}; F[N]\bigr) = \frac{|H|}{|N|} \ind_H^N\Bigl( \chi_\Phi(X_H, \calE_{|X_H}; F[H]) \Bigr) .
  \end{equation*}

   By Frobenius reciprocity a finite dimensional $F[N]$-module $V$ is induced from the $F[H]$-module $W$
   if and only if
  \begin{equation*}
      \dim_F \Hom_{F[N]}(V_\lambda, V) = \dim_F \Hom_{F[H]}\bigl((V_\lambda)_{|H}, W\bigr)
  \end{equation*}
   for all $\lambda \in \widehat{N}_F$. We use this principle in the Euler characteristic.
   For $\lambda \in \widehat{N}_F$ we choose some lattice $L_\lambda \subset V_\lambda$ as in condition \cFtwo.
   We obtain the $N$-sheaf $\Hom_R(L_\lambda, \calE_{|X_H})$ and the $N/H$-sheaf $\Hom_{R[H]}(L_\lambda, \calE_{|X_H})$.
   For simplicity denote the canonical map $X_H \to X_H/N$ by $\pi$ as well.
   Now we  obtain
   \begin{align*}
       \chi\Bigl( \Hom_{F[N]}\bigl(V_\lambda, \com{H}_\Phi(X_H, &\calE_{|X_H})\otimes_R F\bigr) \Bigr) \\
      &= \chi\Bigl( \Hom_{F[N]}\bigl(V_\lambda, \com{H}_\Phi(X_H/N,\pi_*(\calE_{|X_H}))\otimes_R F\bigr) \Bigr) \\
    &= \chi\Bigl(  \com{H}_\Phi(X_H/N,\Hom_R\bigl(L_\lambda, \pi_*(\calE_{|X_H})\bigr))^N\otimes_R F \Bigr) \\
    (\text{by Lemma \ref{lem:FixInside}})\quad &= \chi\Bigl(  \com{H}_\Phi(X_H/N,\pi^N_*\Hom_R\bigl(L_\lambda, \calE_{|X_H} \bigr))\otimes_R F) \Bigr)\\
   &= \chi\Bigl(  \com{H}_\Phi(X_H/N,\pi^{N/H}_*\Hom_{R[H]}\bigl(L_\lambda, \calE_{|X_H} \bigr))\otimes_R F) \Bigr)\\
    (\text{by Lemma \ref{lem:VerdierLemma} and \cFtwo}) \quad &= \frac{|H|}{|N|}\chi\Bigl( \com{H}_\Phi(X_H,\Hom_{R[H]}\bigl(L_\lambda, \calE_{|X_H} \bigr))\otimes_R F) \Bigr)\\
    (\text{by Lemma \ref{lem:FixInside}}) \quad &= \frac{|H|}{|N|}\chi\Bigl( \Hom_{F[H]}\bigl((V_\lambda)_{|H},\com{H}_\Phi(X_H, \calE_{|X_H} )\otimes_R F\bigr) \Bigr).
   \end{align*}

This proves the first equality in the theorem.
Next we decompose the $F[H]$-module $\com{H}_\Phi (X_H, \calE_{|X_H} ) \otimes_R  F$
into isotypical components. Let $V_\lambda$, with $\lambda \in \widehat{H}_F $, be an irreducible module and choose
some $R[H]$ stable lattice $ M_\lambda \subset V_\lambda $.
Put $ D_\lambda := \Hom_{F[H]} (V_\lambda, V_\lambda)$. Then
$ D_\lambda$ is a division algebra and the multiplicity $m^j_\lambda$ of $V_\lambda$ in
$H^j_\Phi(X_H, \calE_{|X_H}) \otimes_R F  $ equals
\begin{align*}
m^j_\lambda & = \dim_{D_\lambda} \Hom_{F[H]} \bigl(V_\lambda, H^j_{\Phi} (X_H, \calE_{|X_H} ) \otimes_R F \bigr) \\
	  & =  ( \deg V_\lambda)^{-1} \dim_F \Hom_{F[H]} ( V_\lambda , H^j_{\Phi} (X_H, \calE_{|X_H} ) \otimes_R F )   \\
	  & =  (\deg V_\lambda)^{-1} H^j_\Phi (X_H, \Hom_{R[H]} (M_\lambda , \calE_{|X_H} )) \otimes_R F  .
\end{align*}
Recall that $\deg V_\lambda = \dim_F(D_\lambda)$.
\end{proof}

\section{Further comments}\label{sec:Comments}

We add some remarks in order to clarify some assumptions for the main result.

\subsection{The finiteness condition}
If $\calE$ is a constant sheaf on $X$, then condition \cFone implies \cFtwo. In general this need not be the case.
This can already be seen in examples where the group acts trivially on the space.
Let $X$ be the unit disc in $\bbC$ and $G = \bbZ / 2 \bbZ$. Then there exists a
sheaf $\calE$ of $\bbZ[G]$-modules on $X$ such that $\com{H}(X, \calE)$ is finitely generated but $ H^2(X, \calE^G)$
is not finitely generated as $\bbZ$-module, i.e.~\cFone holds but \cFtwo fails.

It follows from a \v{C}ech cohomology argument that if $X$ and all $X^H$ are compact and homology locally connected (HLC),
then condition \cF holds for every locally 
constant $G$-sheaf $\calE$ with finitely generated stalks (for the family of all supports).

\subsection{Sheaves of vectorspaces}
If we replace $R$ by $F$ and start with a $G$-sheaf $\calE_F$ of $F$-vectorspaces over $X$, then both
sides of the formula in the theorem make sense under suitable cohomological finiteness assumptions.
However, there are examples which show that the theorem does not hold in this situation.
For instance, let $X = S^1$ be the unit circle with the nontrivial action of $ G = \bbZ / 2 \bbZ $ by rotations.
There is a $G$-sheaf of $F$-vectorspaces $\calE_F$ such that $\com{H}(S^1,\calE_F)$ is finite dimensional and 
$\chi(S^1, \calE_F )=1$. Then clearly 
%\begin{equation*}
$| G | \chi ( X/G , \pi^G_* \calE_F ) \neq \chi (X, \calE_F )$
%\end{equation*}
since the left hand side is an even number.
In fact, most complications in the proof arise from the fact that we have to work with sheaves over the ring $R$. 

\subsection{Cohomology of arithmetic groups}
We indicate that the assumptions of the theorem hold for the cohomology of arithmetic
groups. Let $A$ be a reductive algebraic group defined over $\bbQ$ ($A$ is not necessarily
connected). By $\Gamma \subset  A (\bbQ)$ we denote an arithmetic group.
Assume that $G \subset \Aut_\bbQ(A)$ is a finite
subgroup which acts on $\Gamma$ and on a finite dimensional rational representation $ \rho $ of $A$
on a $\bbQ$-vectorspace $ E $ such that 
 $g ( \rho (\eta) e) = \rho  ( g(\eta)) ge$
for all $e \in E$, $\eta \in A(\bbQ)$, $g \in G$.
Let $ Y $ be the symmetric space attached to $ A (\bbR)$. Then $ G $ and $ \Gamma $ act on $Y$. Put
$ X := \Gamma \backslash Y $ as topological quotient and denote by $ f : Y \longrightarrow X $ the natural projection.
We choose a $G$ and $\Gamma$-stable lattice $ L $ in $ E $. Put $ \calE = f ^\Gamma_\ast L_Y $, where $ L_Y $
is the constant sheaf with stalks $ L$ on $Y$. Then $ G $ acts on $ \com{H} ( X, \calE) \otimes_{\bbZ} \bbQ
= \com{H}(\Gamma, L) \otimes_{\bbZ} \bbQ $. To see that the assumptions of the theorem hold one uses the
Borel-Serre compactification for $ X $ and $X^H$, see \cite{BorelSerre1973}.

\subsection{} 
For a paracompactifying family of supports $\Phi$ on $X$ there is in general no equality of the form
\begin{equation*}
    \com{H}_\Phi(X,\calE)\otimes_R F = \com{H}_\Phi(X,\calE \otimes_R F)
\end{equation*}
for a sheaf of $R$-modules $\calE$ on $X$. For
the family of compact supports this equality holds.
For cohomology  of arithmetic groups, we have for all $R[\Gamma]$-modules $M$ and all $j$ that
$
H^j(\Gamma, M) \otimes_R F = H^j (\Gamma, M \otimes_R F)  \:.$
This follows from the existence of a resolution $P_\bullet \longrightarrow \bbZ \longrightarrow 0 $
of $\bbZ $ by finitely generated free $\bbZ [\Gamma]$-modules and since $F$ is flat as $R$-module, see  Thm.~11.4.4 in
\cite{BorelSerre1973} and p.~193 in \cite{BrownBook1982}.

\providecommand{\bysame}{\leavevmode\hbox to3em{\hrulefill}\thinspace}
\providecommand{\href}[2]{#2}

\end{document}